\documentclass[12pt]{article}
\usepackage{amsmath,amsthm,amsfonts,amssymb,latexsym,enumerate,graphicx,mathrsfs,mathabx,cancel, xcolor,soul}
\usepackage[utf8]{inputenc}
\usepackage{hyperref}
\hypersetup{colorlinks, citecolor=blue, linkcolor=red}

\usepackage{tikz}
\usetikzlibrary{arrows,decorations.pathmorphing,decorations.pathreplacing}
\usepackage{tikz-cd}
\tikzset{
  math to/.tip={Glyph[glyph math command=rightarrow]},
  math double/.tip={Glyph[glyph math command=twoheadrightarrow]},
  loop/.tip={Glyph[glyph math command=looparrowleft, swap]},
  weird/.tip={Glyph[glyph math command=Rrightarrow, glyph length=1.5ex]},
  pi/.tip={Glyph[glyph math command=pi, glyph length=1.5ex, glyph axis=0pt]},
  pool/.tip={Glyph[glyph math command=looparrowleft]},
}

\newtheorem{thm}{Theorem}[section]
\newtheorem{cor}[thm]{Corollary}
\newtheorem*{cor-no}{Corollary}
\newtheorem*{thm-no}{Theorem}
\newtheorem{prop}[thm]{Proposition} 
\newtheorem{lem}[thm]{Lemma}

\newtheorem{mainthm}{Theorem}

\newtheorem{corintro}[mainthm]{Corollary}
\newtheorem{mainprop}[mainthm]{Proposition}

\theoremstyle{definition}

\theoremstyle{remark}
\newtheorem{remark}[thm]{Remark}

\newtheorem{mainexample}[mainthm]{Example}
\newtheorem{example}[thm]{Example}

\newcommand{\define}{\textit}

\makeatletter
\@tfor\next:=abcdefghijklmnopqrstuvwxyzABCDEFGHIJKLMNOPQRSTUVWXYZ\do{%
	\def\command@factory#1{%
		\expandafter\def\csname cal#1\endcsname{\mathcal{#1}}
		\expandafter\def\csname frak#1\endcsname{\mathfrak{#1}}
		\expandafter\def\csname scr#1\endcsname{\mathscr{#1}}
		\expandafter\def\csname bb#1\endcsname{\mathbb{#1}}
		\expandafter\def\csname rm#1\endcsname{\mathrm{#1}}
		\expandafter\def\csname bf#1\endcsname{\mathbf{#1}}
	}
	\expandafter\command@factory\next
}
\makeatother

\newcommand{\mo}{{-1}}
\newcommand{\onto}{\ensuremath{\twoheadrightarrow}}
\newcommand{\bk}[1]{{\left\langle #1 \right\rangle}}

\newcommand{\torus}[1]{{{\bbT_{#1}}}}

\newcommand{\verts}[1]{{{#1}^{(0)}}}

\newcommand{\edges}[1]{{{#1}^{(1)}}}

\newcommand{\aut}[1]{{\mathrm{Aut}\left(#1\right)}}

\DeclareMathOperator{\Aut}{Aut}

\newcommand{\out}[1]{{\mathrm{Out}\left(#1\right)}}

\newcommand{\ad}[1]{{\mathrm{ad}_{#1}}}

\title{Reducing the conjugacy problem for relatively hyperbolic
  automorphisms to peripheral components}

\author{Fran\c{c}ois Dahmani, Nicholas Touikan}
\begin{document}

\maketitle
\begin{abstract}
 We  give  a  reduction  of  the  conjugacy  problem  among  outer automorphisms of free (and torsion-free  hyperbolic) groups to specific algorithmic problems pertaining to mapping tori of polynomially growing automorphisms.  We explain how to use this reduction and solve the conjugacy problem for several new classes of outer automorphisms. This proposes a path towards a full solution to the conjugacy problem for $\out{F_n}$.
\end{abstract}

\tableofcontents

\section*{Introduction}

\subsubsection*{Context and problem}

The decidability of the conjugacy problem in $\out G$, the outer automorphism group of  a torsion-free hyperbolic group $G$, and in particular the conjugacy problem in $\out{F_n},$ remains  an outstanding conjecture in geometric group theory. While this conjecture has been established in various specific cases, as well as for certain subclasses of outer automorphisms (see \cite{FH, D_tams, Se, FrancavigliaMartino}), the general case remains mostly open.

Given a word metric on $G$, the dynamics of an automorphism $\alpha$ on $G$ distinguish specific elements and their conjugacy classes, namely those that are polynomially growing under $\alpha$. Assuming $G$ is torsion-free hyperbolic, a  theorem of Levitt and Lustig \cite{Lev09} (see also \cite{DK}) implies that polynomially growing elements form a family of subgroups. 

The objective of this work is twofold: firstly, to completely reduce the conjugacy problem in $\out{G}$ to problems relating to the restrictions of automorphisms to polynomially growing subgroups

and secondly to provide new classes of outer automorphisms in which we can decide conjugacy.

It is tempting to see the first objective as an analogue of a reduction for automorphisms of finite dimensional vector spaces: conjugacy classes of matrices in $GL(n,{\mathbb{C}})$ are classically understood throught their restriction to generalised eigenspaces on which they induce certain dynamics. Exponential growth in these dynamics is related to eigenvalues of modulus $\neq 1$.So our reduction appears as an analogue of a reduction to generalised eigenspaces of modulus 1 eigenvalues, i.e. to those eigenspaces where growth is polynomial. We will give examples of applications after stating our main result.

\subsubsection*{Main result}

We now give the terminology and notation to state our main result. The mapping torus of a group $A$ for an automorphism $\beta$,  is $A\rtimes_\beta \langle t \rangle$. Its fiber is $A$, its orientation is the coset $tA$. 

Let $\alpha$ be an automorphism of $G$. For each $A$ among the maximal subgroups of $G$  on which $\alpha$ has polynomial growth, there is a minimal $k_A >0$ and an element $g_A\in G$ such that $A$ is invariant under $\beta= \ad{g_A} \circ \alpha^{k_A}$, and this allows to us consider the mapping torus $A\rtimes_\beta \mathbb{Z}$ induced by $\alpha$, which  encodes the dynamics of $\alpha$ on $A$. We call these subgroups the \emph{maximal polynomial growth sub-mapping-tori of $\alpha$.} 

The configuration of the maximal polynomial growth sub-mapping-tori of the mapping torus of $\alpha$ in $\out G$ is a useful conjugacy invariant. Our result not only confirms this, it actually states that once given solutions to problems that are completely intrinsic to these sub-mapping-tori, it is then possible to independently compute all the remaining information needed to distinguish conjugacy classes in $\out G$. Here are algorithmic properties we will require for our polynomial growth sub-mapping-tori:
 \begin{itemize}
     \item A class of groups, that is closed under taking finitely generated subgroups, is \emph{hereditarily algorithmically} tractable (see \cite{Tou}) if its members are  effectively coherent, with recursively enumerable presentations,  and have uniformly solvable conjugacy problem and generation problem. 
     \item A class of groups has \emph{effective congruences separating torsion} (see \cite{DT_invent_2019}) if there is an algorithm that, given any group $H$ in that class, will output a characteristic finite index subgroup $H_0$ such that $\out{H} \to \out{H/H_0} $ has torsion free kernel. 
    \item A class of mapping tori has solvable \emph{fiber-and-orientation-preserving mixed Whitehead problem}   on a family of tuples of group elements, if the orbit problem for the action of fiber-and-orientation-preserving outer automorphism groups on tuples of conjugacy classes of tuples of elements from this family,  is solvable (see \cite{BV2011}).
 \end{itemize}
 
Although our original motivation was the conjugacy problem in $\out{F_n}$, our proof relies on relative hyperbolicity in such a way that the generalization to arbitrary torsion-free hyperbolic groups is immediate.

\begin{mainthm}\label{thmintro}
Let $G$ be a residually finite torsion-free hyperbolic group, and consider a class $\calO\calA$ of outer automorphisms  of $G$. Let $\calP$ be the collection of all maximal polynomial growth sub-mapping-tori given by the elements in  $\calO\calA$  and let $\calP'$ be the class of their subgroups that arise as the images of edge groups of the  JSJ decompositions of mapping tori of automorphisms in $\calO\calA$. Assume that   $\calP$  and $\calP'$ satisfy the following algorithmic properties:

\begin{enumerate}
    \item\label{it:alg-tract} the groups in $\calP$ have a recursively enumerable set of presentations, and have uniformly solvable conjugacy problem, the groups in $\calP'$ belong to a hereditarily algorithmically tractable class of groups, their small subgroups are finitely generated and it is possible to decide if a group in $\calP'$ coincides with the group that contains it in $\calP$,
    \item\label{it:FOP_IP} the fiber-and-orientation-preserving isomorphy (or isomorphism) problem is solvable for mapping tori in $\calP$, 
    \item\label{it:owski} the groups in $\calP'$ have congruences that effectively separate the torsion, and
    \item\label{it:MWHP} for every mapping torus $P\in \calP$,  the fibre-and-orientation preserving mixed Whitehead problem on the class of generating tuples of the subgroups $P' \in \calP'$, is solvable whenever $P'\leq P$.

\end{enumerate}

Then the conjugacy problem in $\out{G}$ for elements in class $\calO\calA$ is solvable.

\end{mainthm}

The class $\calP$ might seem impenetrable, but actually it consists of mapping-tori of malnormal quasi-convex subgroups of $G$, and if $G$ is free, it consists of mapping-tori of free groups of smaller rank.
The algorithmic properties in Theorem \ref{thmintro} could be simplified by relaxing the distinction between $\calP$ and $\calP'$, see Example \ref{eg:ULS} or the statement of Corollary \ref{cor;ung}. That said, it is also worth noting that in the case of mapping tori of automorphisms of one-ended torsion-free hyperbolic groups, although the groups in $\calP$ need not be coherent (see \cite{Rips}), the groups in $\calP'$ turn out to be virtually $\mathbb{Z}^2$. It is this important case that has guided our formulation of the main result.

At the end of this paper in Section \ref{sec;non-growing} we show the required algorithmic properties are satisfied in the following most basic situation. 

\begin{mainprop}\label{prop;intro}
  If $\calP$ consists of groups that are direct products of a finitely generated free group with $\mathbb{Z}$, and if $\calP'$ consists of finitely generated subgroups of groups in $\calP$, then the four algorithmic properties given in Theorem \ref{thmintro} are satisfied for $\calP$ and $\calP'$. 
\end{mainprop}

As we'll see, even this seemingly trivial case has applications.

\subsubsection*{Applications}
The following are examples of groups and classes of outer automorphisms in which it is possible to decide conjugacy using our main result and Proposition \ref{prop;intro}. Some of these examples are already well-known whereas others are new even for free groups.

\begin{mainexample} (Atoroidal automorphisms, well known).
  If $G$ is a free group, and  $\alpha$ is atoroidal, then by \cite{Brinkmann}, all non-trivial elements of $G$ have exponential growth, the only maximal polynomial growth sub-mapping-tori of $\alpha$ is that of the trivial subgroup. Mapping tori with trivial fiber satisfy the four algorithmic properties of Theorem \ref{thmintro}, which solves the conjugacy problem for the class $\calO\calA$ of atoroidal automorphisms. This case was the object of \cite{D_tams}.
\end{mainexample}

\begin{mainexample} (Exponential-or-pointwise-inner automorphisms of free groups, new).
    Recall that  in free groups (and torsion free hyperbolic groups), any pointwise inner automorphism is inner \cite{BV2011}.  If $\alpha$ is an automorphism of free group that is pointwise inner on its polynomially growing elements, then its maximal polynomial growth sub-mapping-tori are direct products of f.g.-free groups with $\mathbb{Z}$.  By Proposition \ref{prop;intro} the four required algorithmic properties are satisfied in this case.
    
    This applies in particular if $G$ is a free group, with a free factor system $G=A_1*A_2*\dots*A_k*F$ and  $\alpha$ has the property that $\alpha^k(g)$ is conjugate to $g$  (for $k\geq 1$) if and only if $g$ is conjugate into one of the $A_i$.  In fact, such $\alpha$ is easily seen to be  relatively atoroidal (i.e. no elements that are hyperbolic w.r.t. the free factor systems are mapped via a power of $\alpha$ to their conjugates) and has no twin pair of subgroups (no pairs of conjugates of the $A_i$ are sent on conjugates by a same conjugator), so \cite{DLi} can be applied to identify the polynomially growing subgroups.
    
\end{mainexample}

\begin{corintro}\label{cor:corintro}
    The conjugacy problem for outer automorphisms of free groups that are pointwise inner on their polynomially growing elements, is solvable. 
\end{corintro}

\begin{mainexample}(Almost toral relatively hyperbolic outer automorphisms of free groups, new).  Brinkmann's theorem \cite{Brinkmann} says that the  automorphisms of a free group that are atoroidal are exactly those producing a word-hyperbolic mapping torus.  Consider now those producing mapping tori that are almost-toral relatively hyperbolic: those are those that are exponentially growing on all elements except on a collection of invariant (up to conjugacy) cyclic subgroups. Their polynomial submapping tori are either isomorphic to $\mathbb{Z}^2$ or to $\bbZ\rtimes\bbZ$.    Since the four algorithmic properties given in Theorem \ref{thmintro} are solvable for the class that consists of $\bbZ\rtimes\bbZ$, the fundamental group of the Klein bottle and its subgroups (see \cite{DFMT}), one can apply   Theorem \ref{thmintro}  to solve the conjugacy problem within the class of so-called almost toral relatively hyperbolic outer automorphisms of $F_n$.  This class contains (but is much larger than) the class of automorphisms arising from pseudo-Anosov mapping classes of surfaces with boundary, for which the conjugacy problem is well known.

\end{mainexample}

The next example is studied in a subsequent work.
\begin{mainexample} (Polynomial part that is unipotent of linear growth)\label{eg:ULS}
In \cite{DT23} the authors show that if the class $\calP$ consists of mapping tori of unipotent linearly growing free group automorphisms, then $\calP$ itself is algorithmically tractable, has effective congruences separating torsion, and has solvable fibre and orientation preserving mixed Whitehead problem for arbitrary finite tuples of tuples. This readily implies that the algorithmic properties required in Theorem \ref{thmintro} are satisfied and therefore that if  $G=F_n$, the free group of rank $n$, and $\calO\calA$ is the class of mixed growth outer automorphisms whose polynomial parts are unipotent linear, then the conjugacy problem is solvable within $\calO\calA$. This result is the base case ($n=1$) of a program to solve the conjugacy problem among automorphisms in $\out{F_n}$ with unipotent polynomially growing part.

\end{mainexample}

We now discuss examples that are not about automorphisms of free groups.

\begin{mainexample}\label{eg:unip-mapping-class} (Some mapping classes of surfaces,  well known).
  If $G$ is a closed orientable surface group $\Sigma$ of genus $g\geq 2$, and $\alpha$ is an automorphism of $G$, induced by a mapping class $\phi$ on $\Sigma$. Assume $\phi$ is of infinite order. If it is pseudo-Anosov, then it only the trivial element is polynomially growing and the assumptions of Theorem \ref{thmintro} are trivially satisfied.     
  
  If it is not pseudo-Anosov, it is reducible: it preserves a maximal disjoint collection of closed simple curves, and some power of $\phi$ preserves the associated subsurfaces. On some of these subsurfaces it induces a pseudo-Anosov mapping class, on the other ones it induces a finite order outer automorphism. The polynomial growth submapping tori are given by, first,  simple curves between two subsurfaces on which $\phi$ is pseudo-Anosov, and second, by the components of the complement of these pseudo-Anosov subsurfaces  (i.e. subsurfaces obtained by glueing along their common boundary curves the elementary subsurfaces on which $\phi$ has finite order). These latter ones are the Seifert-fibered pieces of the classical 3-manifold JSJ decomposition of the mapping torus of $\Sigma$ by $\phi$. On the group theoretical level, these are exactly the polynomial growth sub mapping-tori that are not virtually $\bbZ^2$.

  For the sake of the example, assume that the collection of disjoint simple closed curves is not properly permuted by $\phi$ and that  the finite order mapping classes induced by $\phi$ on subsurfaces are trivial, then we are exactly in the case of Proposition \ref{prop;intro} since surfaces with boundary have free fundamental group.
\end{mainexample}

\begin{mainexample} (Certain automorphisms of residually finite one-ended torsion-free hyperbolic groups, new).
    Let $G$ be a one-ended residually finite torsion-free hyperbolic group and let $\alpha$ be an automorphism that is pointwise inner on its polynomially growing part. Then by \cite{Lev05} it is possible to refine the JSJ decomposition of $G$ in the QH subgroups so that certain vertex groups of the refinement are $\pi_1$-images of compact subsurfaces $\Sigma_1,\ldots,\Sigma_p$ of the surfaces underlying the QH groups of the JSJ decomposition of $G$ on which $\alpha$ induces pseudo-Anosov mapping classes. We call such vertex groups \emph{active}. $\alpha$ restricts to a trivial outer automorphism on the subgroups of $G$ that come from the subgraphs of groups coming from the connected components of the complement of the active vertices. These (fundamental groups of) subgraphs of groups $G_i$ are the maximal  polynomially growing subgroups and the mapping tori are seen to be direct products of  the form $G_i\times \bbZ$. The edge groups of JSJ decomposition of the mapping torus $G\rtimes \langle t \rangle$ are copies of $\bbZ^2$ that have one generator in the fiber $G$ and another generator in the coset $tG$. Since edge groups are always $\bbZ^2$ and the $G_i$ are all torsion-free hyperbolic with computable presentations algorithmic properties \ref{it:alg-tract}, \ref{it:owski}, follow immediately from known facts about $\bbZ^2$, hyperbolic groups, and direct products with $\bbZ$. Within every maximal trivial growth mapping torus $P \in \calP$ the image of an edge group can be given by a  generating set of the form $\langle{g_e,t_P}\rangle$ where $t_P$ is central in $P$ and picked out from $\{t_p^{\pm 1}\}$ by the orientation and $g_e$ lies in the fiber $G_i$. In particular algorithmic property \ref{it:MWHP} follows from the solution to the isomorphism problem for hyperbolic groups relative to cyclic groups (see \cite{DGr,DG_gafa_2010}) or directly from \cite[Theorem 6.1]{BV2011}, which also implies that property \ref{it:FOP_IP} is satisfied. Thus we have:
\end{mainexample}

\begin{corintro}\label{cor:hyp-triv-ip}
    The conjugacy problem for outer automorphisms of one-ended  torsion-free hyperbolic groups that are pointwise inner on their polynomially growing elements, is solvable. 
\end{corintro}

    Although, given \cite{Lev05}, this example seems tame, the machinery we are using, particularly the mixed Whitehead problem on generating tuples of edge groups, is exactly what is needed to handle the case where two outer automorphisms $[\alpha],[\beta]$ that are pointwise inner on their polynomially growing parts are conjugate by an automorphism that induces a non-trivial automorphism of the graph underlying the JSJ decomposition. The next is an example involving a many-ended non-free hyperbolic group.

\begin{mainexample} (Free products of free and surface groups, new, see Figure \ref{fig:auto-eg})
  If $G$ is a free product of free groups and closed surface groups, and $\alpha$ is an automorphism of $G$ that restricts to a pseudo-anosov an automorphisms of the closed surface free factors and all fixed subgroups are cyclic. Once again, the four algorithmic properties of our main result are solvable in the class of maximal polynomial growth sub-mapping-tori of $\alpha$ since these will be virtually $\bbZ^2$. $\alpha$ could also restrict to automorphisms such as the ones given in Example \ref{eg:unip-mapping-class}, in which case Proposition \ref{prop;intro} also applies.
\end{mainexample}

\begin{figure}
    \centering
    \includegraphics[width=0.5\textwidth]{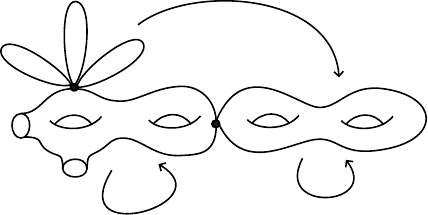}
    \caption{An automorphism carried by a homotopy equivalence that restricts to pseudo-Anosov maps on surfaces with or without boundary. A remaining free factor is not preserved and all its elements grow exponentially.}
    \label{fig:auto-eg}
\end{figure}

\subsubsection*{On the proof}

The conjugacy problem for outer automorphisms of a group $G$ is equivalent to the fibre-and-orientation-preserving isomorphy problem for mapping tori of $G$ given by the considered automorphisms. 

Our theorem will be obtained via the relative hyperbolicity of the semidirect products $G\rtimes \mathbb{Z}$. This geometric feature, suggested by Gautero and Lustig \cite{GL}, and proved by Ghosh \cite{Ghosh},  Li and the first author \cite{DLi}, lets us apply the body of work on the isomorphy (or isomorphism) problem for relatively hyperbolic groups \cite{DT_invent_2019, DG_Duke}. However, the constraint of preserving the fiber and the orientation is important and difficult.  

The theory of Dehn fillings of relatively hyperbolic groups will be a crucial ingredient in two steps of the argument. 

It first appears as co-finite Dehn fillings of the peripheral edge groups of a JSJ decomposition, which allow us to compute the finite groups of outer-automorphisms that may twist the glueing maps in this JSJ decomposition. These are the Dehn fillings that appear in \cite{DT_invent_2019, DG_Duke}. Unfortunately, the fibered structure of our mapping tori does not persist in these Dehn fillings. In particular this technique alone is insufficient  to solve the fiber and orientation preserving isomorphism problem.

A crucial innovation in this paper is the use of a second sequence of Dehn fillings that only fill-in subgroups of the fiber in such a way that the parabolic subgroups of the mapping torus are mapped to virtually cyclic groups. The quotient is therefore word hyperbolic and furthermore it is fibered. Using these co-virtually-cyclic Dehn fillings, and arguments involving the small modular group of the quotient, it is then possible to decide the existence of fiber and preserving isomorphisms between the the given mapping tori.

{\it Acknowledgments.} We would like to thank the referees for useful comments. We also thank Armando Martino and Stefano Francaviglia for encouraging discussions around the methods and the applications. The second named author is supported by an NSERC Discovery grant.

\section{Preliminaries}

 \subsection{Graphs, graphs of groups, and  their fundamental groups}\label{sec;vocab}

A graph $X= (\verts X, \edges X, \iota, \tau, -)$ is  a set of vertices
$\verts X$, a set of
oriented edges $ \edges X$, endowed with two maps $\iota:\edges X \to
\verts X $, $\tau:
\edges X \to \verts X $ and a fixed-point free involution $-: \edges X\to \edges X$ satisfying
${\iota(\bar e)} =
\tau(e)$.
We will freely use the geometric realisation of a graph, and the vocabulary from it.

 A graph of groups $\bbX$ is 
 given by  a graph $X$, 
  a group $G_v$ for each vertex $v$,
  a group $G_e = G_{\bar e}$ for each unoriented edge $\{e, \bar e\}$, 
  and an injective homomorphism $t_e  : G_e \to
   G_{\tau(e)}$ for each oriented edge $e$ of   terminal 
 vertex
   $\tau(e)$.

   The Bass group of $\bbX$ is the group generated by the collection
   of all vertex groups $G_v, v\in X^{(0)}$, and the collection of all
   edges $e\in X^{(1)}$, subject to the relations that, for all $e$,
   $e\bar e = 1$ and that, for all $e$ and all $g\in G_e$,
   $\bar e t_{\bar e}(g) e = t_e(g) $.  The generators corresponding
   to edges are called Bass generators.

     Given $v_0$, a vertex of $X$, the \define{fundamental group}
     $\pi_1(\bbX, v_0)$ is the subgroup of the Bass group whose
     elements are given by a word $g_0e_1g_1e_2g_2\dots g_ne_ng_{n+1}$ such that $g_0, g_{n+1}\in G_{v_0}$, $\iota(e_1) = v_0= \tau (e_{n})$, and for all $i>0$, with $i\leq n$, $g_i\in G_{\tau(e_{i})}$ and $\tau(e_{i})=\iota(e_{i+1})$.

    Given  $\tau_0$ a spanning subtree of
    $X$, the
    \define{fundamental group}  $\pi_1(\bbX, \tau_0)$ is the quotient of the Bass
    group obtained by identify all edges in $\tau_0$ to the
    identity.
    
     An important theorem of the theory of Bass-Serre
    decompositions is that this quotient map induces an isomorphism
    from $\pi_1(\bbX, v_0)$ to $\pi_1(\bbX, \tau_0)$.

    \subsection{Automorphisms of graphs of groups}\label{sec;autom_gog}

    In a group, we denote conjugation as follows
    $a^g = g^\mo a g = \ad{g}(a) $. An automorphism $\Phi$ of a graph
    of groups $\bbX$ is a tuple consisting of an autormorphism
    $\phi_X$ of the underlying graph $X$, an isomorphism
    $\phi_v : G_v\to G_{\phi_X(v)}$ for every vertex $v$, an isomorphism
    $\phi_e : G_e\to G_{\phi_X(e)}$ for every edge $e$, with
    $\phi_{\bar e} = \phi_e$, and elements
    $\gamma_e\in G_{\phi_X(\tau(e))}$ for every edge $e$, that satisfy the
    Bass diagram:
   
\begin{equation}\tag{Bass  Diagram} 
\begin{array}{cccccl}
  G_e & \stackrel{t_e}{\longrightarrow} & G_{\tau(e)}  &
                                                   \stackrel{\phi_{\tau(e)}}{\longrightarrow}
  & G_{\phi_X(\tau(e))}  & \\
&&&&& \\
 ||  & & & &  \Big\uparrow   \ad{\gamma_e} \\
&&&&&\\
G_e &    \stackrel{\phi_{e}}{\longrightarrow}   & G_e &
                                                        \stackrel{t_{\phi_X(e)}}{\longrightarrow}
  &  G_{\phi_X(\tau(e))}  & \end{array}                                                
  \end{equation}
    
 which can be also be written as an equation:

 \[ \phi_{\tau(e)} \circ t_e = {\rm ad}_{\gamma_e}\circ  t_{\phi_X(e)} \circ \phi_e.\]

 One can check that an automorphism of the graph of groups $\bbX$
 extends naturally to an automorphism of the Bass group, by sending
 the generator $e$ to $\gamma_{\bar e}^{-1} \phi_X(e) \gamma_e$.
 While this does not necessarily preserves the vertex $v_0$ or the
 spanning subtree $\tau_0$, it induces an isomorphism between
 $\pi_1(\bbX, v_0)$ and $\pi_1(\bbX, \phi_X(v_0))$.  See \cite[\S
 2.3]{Bass}, \cite[Lemmas 2.20-2.22]{DG_gafa_2010}.

The composition rule is as follows (where we replaced the notations $\phi_X, \phi_X'$ by $f, f'$ for readability) \begin{multline*} (f,
  (\phi_v)_v, (\phi_e)_e, (\gamma_e)_e) \circ (f', (\phi'_v)_v,
  (\phi'_e)_e, (\gamma'_e)_e)  = \\  ( f\circ f',  \, (\phi_{f'(v)} \circ
  \phi_v)_v, \,  (\phi_{f'(e)} \circ \phi_e )_e, \, 
  (\phi_{f'(\tau(e))}(\gamma'_e)) \gamma_{f'(e)}).  \end{multline*}

The group of all automorphisms of the form
$(\phi_X, (\phi_v)_v, (\phi_e)_e, (\gamma_e)_e)$ satisfying the Bass
diagram is denoted $\delta \Aut (\bbX)$, and maps naturally to a well
defined subgroup of $\out{ \pi_1(\bbX, v)}$.

The subgroup  $\delta_0 \Aut (\bbX)$ consists of automorphisms of the form \[(Id_X, (\phi_v)_v, (\phi_e)_e, (\gamma_e)_e).\]   If $\Phi_1, \Phi_2$ define the same underlying graph isomorphism, then $\Phi_1^{-1} \Phi_2 \in \delta_0 \aut{\bbX}$. In particular,  $\delta_0 \aut{\bbX}$ is of finite index in  $\delta \Aut (\bbX)$.

\begin{prop}\label{prop;coset_computation}
  Consider a finite graph of groups $\bbX$, whose vertex groups lie
  in a class of groups for which the isomorphism problem and the mixed
  Whitehead problem are solvable. Then there is an algorithm that
  computes a complete set of coset representatives of
  $\delta_0 \Aut (\bbX)$ in $\delta \Aut (\bbX)$.
\end{prop}
\begin{proof}
  Let $\phi_X:X\to X$ be a graph isomorphism. We wish to know whether it
  can be extended into a graph of groups automorphism in
  $\delta \aut{\bbX}$.  By hypothesis, we may determine whether for
  each vertex $G_v$ is isomorphic to $G_{\phi_X(v)}$, and if so, find such
  an isomorphism.  If for one vertex $v$, the vertex groups $G_v$ and
  $G_{\phi_X(v)}$ are not isomorphic, then $\phi_X$ cannot be extended into an
  automorphism of graph of groups. We can therefore produce the finite
  list of these graph automorphisms.

  \cite[Proposition 4.4]{DT_invent_2019} then ensures that the graphs
  of groups are isomorphic if and only if there exists an extension
  adjustment (see \cite[Definition 4.3]{DT_invent_2019}).
 The solution to the mixed Whitehead problem for vertex groups ensures
 that one can decide whether  such an adjustment exists.
\end{proof}

Recall that, following \cite[\S 1.2]{D_tams}, the \define{small
  modular group} of a graph of groups $\bbX$ is the subgroup of $\delta\aut{\mathbb{X}}$whose elements are  of the form
$(Id_X, (\ad {\gamma_v}), (Id_e), \gamma_e)$, for $\gamma_v \in
G_v$. We note that the Bass diagram imposes that
$\gamma_v\gamma_e^{-1} \in Z_{G_{\tau(e)}} (t_e(G_e))$ for all edge $e$.
Recall also that it is generated by Dehn twists over edges of the
graph of groups $\bbX$.

\section{The main result}

\subsection{General statement}\label{sec;as_st}

We now need to detail our vocabulary on mapping tori. Let $F$ be a
group, $\alpha$ an automorphism of $F$, and
$\torus \alpha = F \rtimes_\alpha \langle t_\alpha\rangle$ its mapping
torus.  This means that in $\torus \alpha$, if $a\in F$, then
$\ad{t_\alpha}(a) = \alpha(a)$. We will often write $t$ for $t_\alpha$
if the context allows the abuse. 

The group $\torus \alpha$ is equipped with a fibre and an orientation:
we call $F$ the \define{fibre} of the semi-direct product  $\torus \alpha$, and 
we say that the coset  $F t $ determines the positive \define{orientation} of  the semi-direct
product.

Consider now $\torus \alpha$ and $\torus \beta$ the mapping tori of
the same group $F$, by automorphisms $\alpha$ and $\beta$. 
If $\psi: \torus \alpha \to \torus \beta$ is an isomorphism, we say
that it is \define{fibre preserving} if $\psi(F) = F$, and in this case we say
that it is \define{orientation preserving} if $\psi(t_\alpha) \in F 
t_\beta $.

For  a subgroup $H$ of $F$ whose conjugacy class is preserved by
some power of $\alpha$,  the \define{sub-mapping torus} of $H$ is obtained as
follows. Let $k>0$ the smallest integer such that there exists $a\in
F$ for which  $\alpha^k(H) = \ad{a}(H)$. The  sub-mapping torus of $H$
is the subgroup $\langle H, t_\alpha^ka^{-1}   \rangle$,  of $\torus
\alpha$. It does not depend on the choice of $a$, and is isomorphic to $H\rtimes_{(\ad{a^{-1}} \circ \alpha^k)}
\langle t_\alpha^ka^{-1}\rangle$.

A group is small if it does not contain a non-abelian free subgroup.
We may now state our main result. In its statement, $F$ is not necessary a free group ($F$
stands for Fibre).  We refer to the introduction for the definitions of the algorithmic problems. 

\begin{thm} \label{thm;IP_fop}
	
  Let $F$ be a finitely generated, torsion-free group, and let $\calO\calA$ be a collection of outer automorphisms of $F$. Consider the
  class $\calT$ of mapping tori $\torus \alpha = F\rtimes_\alpha \bbZ$ of $F$, over automorphisms in the classes of $[\alpha] \in \calO\calA$. Assume that each $\torus\alpha$  is relatively
  hyperbolic with respect to a collection $\calP_\alpha$ of residually finite sub-mapping tori. Let $\calP'_\alpha$ be the class of all finitely generated subgroups of groups in $\calP_\alpha$, that arise as edge groups in the JSJ decompositions of groups in $\calT$.  Assume that $\calP' = \bigcup_{[\alpha]\in \calO\calA}  \calP'_\alpha$ and $\calP = \bigcup_{[\alpha]\in \calO\calA} \calP_\alpha$ satisfy: 
  \begin{enumerate}
  \item\label{it:al-tract} the groups in $\calP$   have a recursively enumerable set of presentations   and have uniformly solvable conjugacy problem, the groups in $\calP'$ belong to an hereditarily
    algorithmically tractable class of groups, their small subgroups are finitely generated, and it is possible to decide if a group in $\calP'$ coincides with an overgroup in $\calP$,
  \item the fibre and orientation preserving isomorphism problem is
    solvable for the class $\calP$,
  \item the groups in $\calP'$ have congruences that separate the torsion effectively,
  \item 
    in the class of mapping tori in $\calP$, for each group in $\calP_\alpha$ the fibre-and-orientation preserving mixed Whitehead problem on the family of generating tuples of its subgroups in $\calP'_\alpha$ is solvable.   
  \end{enumerate}
  
  Then in the class $\calT$, the fibre and orientation preserving
  isomorphism problem is solvable, and in the collection $\calO\calA$, the conjugacy problem in $\out{F}$ is solvable.

\end{thm}

\begin{remark}
 Theorem \ref{thmintro} of the introduction follows from this and the relative hyperbolicity proved in \cite{DLi, DK}.
\end{remark}

\begin{remark}\label{rem-z} The assumptions such as of belonging
to an hereditarily algorithmically tractable class of groups in item \ref{it:al-tract}. of Theorem \ref{thm;IP_fop} allow the computation of the canonical JSJ splitting, which includes finding presentations of the vertex. This assumption can be substituted with another way of computing the canonical JSJ splitting, without changing the argument. 
\end{remark}

We call a subgroup of $\torus \alpha$ \define{elementary} if it is either
cyclic or a subgroup of a group in $\calP_\alpha$. For the rest of the
paper we consider $F$, $\torus \alpha$ and $\torus \beta$ as in the
statement of Theorem \ref{thm;IP_fop}, and aim to determine whether
they are isomorphic via a fibre and orientation preserving
isomorphism.  We will sometimes write $t$
for either $t_\alpha$ or $t_\beta$, when the context is clear.

\subsection{JSJ decompositions}

First we record that no splitting of a mapping torus $\torus \alpha$
of a finitely generated, torsion-free group $F$ has a vertex group
that is (QH) -- quadratically hanging (or hanging fuchsian).

\begin{prop}[{\cite[Proposition 2.11]{D_tams}}] \label{prop;noQH} If
  $\torus \alpha$ acts co-finitely on a tree $T$ 
  and if $G_v$ is the stabilizer of $v$, and if $\calP_v$
  is the peripheral structure on $G_v$ given by the collection of
  adjacent edge stabilizers, then $(G_v, \calP_v)$ is not isomorphic
  to the fundamental group of a surface, with the peripheral structure
  of its boundary components.
  
\end{prop}

The given reference covers trees with cyclic edge
groups, we explain why it is sufficient.

\begin{proof} Consider the tree $\dot T$ obtained from $T$ by
collapsing all edges with non-cyclic stabilizer. Arguing by contradiction, if there exists a vertex $v$ of $T$ whose
group $(G_v, \calP_v)$ is isomorphic to the fundamental group of a
surface with the peripheral structure of its boundary components, then  its star would be unchanged in the collapse $\dot T$, and its
image in $\dot T$ would have same stabilizer. However $\dot T$ satisfies the assumptions of {\cite[Proposition 2.11]{D_tams}}, which outrules such stabilizers. This proves the Proposition. 
 \end{proof}

There is a specific $\torus \alpha$-tree that is called the $JSJ$-tree
of the relatively hyperbolic group $(\torus \alpha, \calP_\alpha)$
that is invariant by automorphisms of $\torus \alpha$. We refer the
reader to \cite{GL_JSJ}, and to \cite[Theorem 3.22]{DT_invent_2019}
for a convenient characterisation.

Recall that Bass-Serre theory dualises actions on trees and
decompositions into graphs of groups (or splittings). We call a
splitting essential if there is no valence $1$ vertex, whose group is
equal to its adjacent edge group. The following is an immediate
consequence of \cite[Theorem 3.22]{DT_invent_2019} and Proposition
\ref{prop;noQH} above, which guarantees the absence of QH vertex groups.

\begin{prop}
  The JSJ decomposition of $\torus \alpha$ is the unique essential
  splitting that is a bipartite graph of groups, with white vertices,
  and black vertices, such that black vertices are maximal elementary,
  white vertices are rigid (in the sense that they have  no further compatible elementary splitting),  
  and such that for any two
  different edges adjacent to a white vertex $w$, the image of the two
  edge groups in the vertex group $G_w$ are not conjugate in $G_w$
  into the same maximal elementary subgroup of $G_w$.
\end{prop}

\begin{prop} Assume that the class of peripheral subgroups in
  $\calP_\alpha$ and edge groups in $\calP_\alpha'$ satisfy the requirements of item \ref{it:al-tract} of Theorem \ref{thm;IP_fop}.  Then, there exists an algorithm that, given $\torus \alpha$, computes the JSJ splitting of $\torus \alpha$.
\end{prop}

The algorithm is actually as explicit as the given algorithms in the
assumptions. The proposition follows from \cite[Theorem
3.26]{DT_invent_2019} after observing that $\torus \alpha$ is always
one-ended and torsion-free and that full hereditary algorithmic tractability is only needed for the class of edge groups $\calP_\alpha'$.

\subsection{On white vertices: Congruences and Dehn fillings for lists}

In the following, the graphs of groups $\bbJ_\alpha$, $\bbJ_\beta$ are
the JSJ decompositions of $\torus \alpha, \torus \beta$
respectively. Let $J_\alpha$ and $J_\beta$ the underlying graphs of
these graph of groups.  We assume that we are given an isomorphism of
graphs $\phi_J: J_\alpha \to J_\beta$.

\subsubsection{Peripheral structures,  orbits of markings}

 All our tuples will be ordered (and finite), while sets are not ordered. Let $G$ be a group. A \define{peripheral structure on $G$} is a finite set of
conjugacy classes of subgroups.  An \define{ordered peripheral structure} is a tuple of conjugacy classes of
subgroups. A \define{marked
  peripheral structure} is a finite set of conjugacy classes of
 tuples of elements, while a \define{marked ordered peripheral structure} is a tuple of conjugacy classes of tuples of elements.  For each marked structure, there is an \define{associated unmarked} structure, by taking the  subgroups generated by the tuples of elements.

We sometimes abuse terminology by
saying that a peripheral structure  is the (usually infinite) set of all subgroups whose
conjugates belong to the given finitely-many conjugacy classes.

Let $G$ be a group endowed with a peripheral structure $\calE$, and $G'$ another group.

We say that two isomorphisms $\phi, \psi: G \to G'$ \define{peripherally
coincide on a  peripheral structure} $\calE$ if for every
subgroup $E \in \calE$, there exists $h=h(E) \in G'$ such that the
restriction $\psi|_{E}$ differs from $\phi|_{E}$ by the
(post)-conjugation by $h$ in $G'$.

\subsubsection{White vertices of the JSJ decompositions and peripheral structures}

We apply this setting to the white vertices of the JSJ decompositions.

Consider a white vertex $w$ of the JSJ decomposition $\bbJ_\alpha$.
We endow its group $G_w$ with several peripheral structures.  First
there is $\calP(w)$ the relatively hyperbolic peripheral structure
coming from that of the ambient relatively hyperbolic group: it
consists of the subgroups that are intersection of $G_w$ with maximal
parabolic subgroups of $\torus \alpha$ in $\calP_\alpha$.

Second there is $\calE_\calZ(w)$ a cyclic peripheral structure coming from the cyclic edge groups that are not subgroups of groups in $\calP(w)$.  It is the collection of conjugacy classes of subgroups of $G_w$ that are
cyclic, conjugate to an adjacent edge group in $\bbJ_\alpha$, and not
subgroups of groups in $\calP(w)$. We say that this peripheral
structure is \define{transverse (to $\calP(w)$).}  

Observe that these groups intersect trivially the fibre:  by finiteness of the graph of groups, and invariance by $t_\alpha$, it  needs to be a semi-direct product of its intersection with the fibre by $\mathbb{Z}$.

Third, there is the non-transverse edge peripheral structure $\calE_\calP(w)$ consisting  of edge groups that are subgroups of groups in  $\calP(w)$: they are parabolic in the relatively hyperbolic structure. Observe that all edge groups, which are cyclic or parabolic, are in one of the peripheral structures $\calE_\calP(w)$, or  $\calE_\calZ(w)$.

These three peripheral structures are unmarked but ordered (by
choice of an order on the set of edges of $J_\alpha$ around each vertex): they correspond to an ordered
finite family of conjugacy classes of subgroups, but they do not correspond to classes of
generating sets of these subgroups.  However, $\calE_\calZ(w)$ is
naturally marked: each of its subgroups has a unique generator with
positive orientation.

We denote by $\calE_w$ the ordered concatenation of the peripheral structures
$\calE_\calZ(w)$ and $\calE_\calP(w)$. It is the peripheral structure
of adjacent edge groups of $G_w$. The group $\out{G_w, \calE_w}$ acts
on the set of markings of $\calE_\calZ(w)$ and $\calE_\calP(w)$

\subsubsection{Dehn fillings}

We need now to discuss Dehn fillings, Dehn kernels, and the
$\alpha$-certification property, given in \cite{D_MSJ}.

Given a relatively hyperbolic group $(G,\calP)$, with a choice of
conjugacy representatives  $\{P_i\}$ of the peripheral structure
$\calP$, a \define{Dehn kernel} is a normal subgroup $K$ of $G$ normally
generated by  subgroups $N_i$ of the peripheral groups $P_i$: in other words $K = \langle \langle \bigcup N_i \rangle\rangle_G$.  The quotient of $G$ by the Dehn kernel $K$ is a \define{Dehn filling of $G$.}

The Dehn filling theorem \cite{Osin_filling} states that there exists
a finite set $S$ of $G\setminus\{1\}$ such that, whenever $N_i \cap S$
is empty for each $i$, the group $G/K$ is hyperbolic relative to the
injective images of $P_i/N_i$.

If one is given a quotient $G \onto \mathbb{Z}$, we call its kernel
$F$ a fibre for $G$, and we say that a Dehn kernel is \define{in the
  fibre (or is fibered) with respect to this quotient}, if each $N_i$
is contained in $F$. We will say a Dehn
filling \define{fibered} if it is obtained by quotienting by a fibered
Dehn kernel.

When $G$ maps onto $\bbZ$ as before, with fibre $F$,  two sequences of Dehn kernels will be of interest to us. 

First, for each $m \geq0$, the kernel $K^{(m)}$  is defined by setting $N_i^{(m)}$ to be the intersection of all index $\leq m$ subgroups of $P_i$; observe that they are not fibered Dehn kernels.  We will denote $G/K^{(m)}$ by $\overline{G}^{(m)}$. 

Second, for each $m \geq0$, the kernel $K^{(m,f)}$ is defined by
setting $N_i^{(m,f)}$ to be the intersection $N_i^{(m)} \cap F$, where
$F$ is the given fibre. We will denote $G/K^{(m,f)}$ by
$\overline{G}^{(m,f)}$.  These Dehn kernels are thus fibered. Observe
that the Dehn filling theorem gives the following (recall that peripheral subgroups of finitely generated relatively hyperbolic groups are finitely generated).

\begin{prop} If $G$ is a finitely generated  relatively hyperbolic group, that is residually finite,  with
  a fibre $F$, then for every sufficiently large $m$, the group $G/
  K^{(m)}$ is  hyperbolic relative to the subgroups
  $P_i/N_i^{(m)}$, which are finite, and  the group $G/ K^{(m,f)}$ is
   hyperbolic relative to the subgroups $P_i/N_i^{(m,f)}$,
  which are virtually cyclic.

	 In both cases,  $G/ K^{(m)}$ and $G/ K^{(m,f)}$  are word-hyperbolic. Moreover, in the second case, the image of $P_i/N_i^{(m,f)}$ in $G/ K^{(m,f)}$  is a subgroup that is its own normaliser.
\end{prop}

Only the last conclusion requires an explanation: any infinite maximal
parabolic subgroup of a relatively hyperbolic group is its own
normaliser, by \cite[Thm
1.14]{Osin}. The Dehn filling theorem actually says that the image of $P_i/N_i^{(m,f)}$ is a maximal parabolic subgroup of a relatively hyperbolic structure and it is infinite. The conclusion follows.

Consider a white vertex $w$ of the JSJ decomposition $\bbJ_\alpha$,
and its group $G_w$, and its relatively hyperbolic peripheral
structure $\calP(w)$.

The considerations above apply to $(G_w, \calP(w))$  since it is residually finite, relatively hyperbolic, and endowed with a natural quotient onto $\bbZ$.

For fibered Dehn fillings, we must introduce some conditions and find Dehn fillings that satisfy them. In particular we introduce the $\alpha$-certification of Dehn fillings, already used in \cite{D_MSJ}, that will permit, first the use of \cite[Prop. 2.4]{D_MSJ} (that we can collect a list of automorphisms for which the images of edge groups through Dehn fillings is well-behaved), and  later, to prove   Proposition \ref{prop;output_controlled_in_vertices} (through Lemma \ref{lem;216}) that collects a list of so-called fibre-controlled automorphisms (see the definition before the said Proposition).

We call a Dehn filling $\overline{G_w}^{(m,f)}$ a \define{resolving Dehn filling}  if all the following conditions are fulfilled:
\begin{itemize}
\item the quotient is hyperbolic and rigid (in the sense that it has
  no elementary splitting and it is not a virtual surface group), and for any $P_i$ in $\calP(w)$, its image  in  $\overline{G_w}^{(m,f)}$ is naturally isomorphic to $P_i/N_i^{m,f}$, 
\item and the Dehn filling is $\alpha$-certified as defined in \cite{D_MSJ}, that is:  for each (cyclic) subgroup $C$ in $\calE_\calZ(w)$, the centraliser of the image of $C$ in the
quotient is equal to the image of the centraliser of $C$. 
\end{itemize}

Observe that the kernel of the quotient map $G_w\to \overline{G_w}^{(m,f)}$ is in the fiber, hence the peripheral
  subgroups in $\calE_\calZ(w)$ (that are transverse to $\calP(w)$) map injectively in the quotient.

\begin{lem} If all small subgroups of the fibres of the  peripheral subgroups of $G_w$ are finitely generated, then there is an algorithm that, given a resolving Dehn
  filling of $G_w$, terminates and produces a proof that it is resolving.
  Moreover, for all $m_0$, there are resolving Dehn fillings of $G_w$
  obtained by a choice of $m$ larger than $m_0$.
\end{lem}

\begin{proof}
  By the Dehn filling theorem, for sufficiently deep Dehn
  fillings, the quotient is hyperbolic relative to virtually cyclic
  subgroups, hence word-hyperbolic, and the peripheral quotients
  inject. By \cite[Lemma 2.12]{D_MSJ}, sufficiently deep Dehn fillings
  satisfy the $\alpha$-certification property. 
  
  Let us argue that they
  are rigid. For that we will show that they have no peripheral
  splitting (i.e. no splitting relative to the peripheral structure
  $\calP(w) \cup \calE_{\calZ}(w)$, over some parabolic subgroups),
  and no splitting relative to their parabolic peripheral structure,
  over a maximal cyclic subgroup. We need a short digression in order
  to use Groves and Manning's result  \cite[Theorem 1.8]{GM_DFandES}.
        Any small subgroup $U$ of $G_w$ has to intersect the fibre as a
        small group, hence in the peripheral structure. By assumption, this intersection is finitely generated. 
        Therefore, $U$  is finitely generated, since it
        is (f.g. small)-by-cyclic. Also, the group $G_w$, as a semidirect
        product of a finitely generated group with $\mathbb{Z}$, is
        one-ended.  We can therefore use \cite[Theorem
        1.8]{GM_DFandES} to conclude that sufficiently deep Dehn
        fillings have no peripheral splittings, for the peripheral
        structure $\calP(w) \cup \calE_{\calZ}(w)$.  Finally, in order
        to check that they don't have splittings relative to the 
        peripheral structure $\calP(w) \cup \calE_{\calZ}(w)$ over a
        maximal cyclic subgroup, we invoke \cite{DG_Duke}: if a
        sequence of such Dehn fillings all had such a splitting,
        considering a diagonal sequence of Dehn twists over these
        splittings, one gets a contradiction to \cite[Corollary 5.10 (of Proposition 5.8)]{DG_Duke}, applied to $G'=G=G_w$.
        Therefore, sufficiently deep Dehn fillings of $G_w$ are
        rigid. 
        
        All the properties are algorithmically verified, by
        \cite{Papasoglu} (for word hyperbolicity), \cite[Corollary
        3.4]{DG_gafa_2010} (for rigidity), and \cite[Lemma
        2.8]{DG_gafa_2010} (for the computation of centralizers, and
        identification of the images of the peripheral quotients).
\end{proof}

\subsubsection{Back to lists for white vertices}

We come back to our context, in which $w$ is a white vertex  of the canonical JSJ decomposition of $\torus \alpha$, and $w'$ is its image by some graph isomorphism $\phi_J$.  

Observe that the groups $G_w$ and $G_{w'}$ are finitely presented, relatively hyperbolic with respect to the peripheral structures $\calP(w), \calP(w')$, which consist of infinite residually finite groups, and admit no peripheral splitting over an elementary group.  We may therefore apply \cite[Propostion 5.1]{DT_invent_2019} in order to obtain the following. 
 
\begin{prop}[See {\cite[Proposition
  5.1]{DT_invent_2019}}] \label{prop;orbit_m}

    Assume that congruences effectively separate the torsion in the peripheral subgroups of $G_w$.
  
	Given a marking of $\calE_\calZ(w) \cup \calE_\calP(w)$, its orbit under the group $\out{G_w, \calE_w}$ is computable.
\end{prop}

In order to prove this Proposition, Dehn fillings were used such that
the corresponding Dehn kernels were finite index subgroups of the
maximal parabolic subgroups, and, in particular, it is important that
these subgroups are chosen deep enough to ensure that Dehn fillings 
are congruences that separate the torsion (in $\out{G_w}$). We do not give the detail of the argument, since it
is covered by the case treated in \cite[Proposition 5.1]{DT_invent_2019}.

In the next Proposition, we will use Dehn fillings with fibered Dehn
kernels (i.e. lying in the fibre). It will not be important here whether or not there are congruences that separate the torsion.

\begin{prop} \label{prop;double_DF}  Consider $w$  a white vertex in the graph of groups $\bbJ_\alpha$, and $w' = \phi_J(w)$ a white vertex in the in graph of groups $\bbJ_\beta$.      Consider the groups $G_w, G_{w'}$ with their  (unmarked ordered)    peripheral structures  $\calP(w)$,  $\calE_\calZ(w)$ and  $\calE_\calP(w)$, and      $\calP(w')$,  $\calE_\calZ(w')$ and  $\calE_\calP(w')$.    Choose a marking     $(\calE_\calZ(w))_\frakm, (\calE_\calP(w))_\frakm$.   
	
	\begin{itemize}
\item 	It is decidable whether there exists a fibre-and-orientation preserving isomorphism from $(G_w, \calE_\calZ(w), \calE_\calP(w))$ to $(G'_{w'}, \calE_\calZ(w'), \calE_\calP(w'))$. 
	
\item If there exists such an isomorphism, let
  $ (\calE_\calZ(w'))_{\frakm}, (\calE_\calP(w'))_\frakm$ be the image
  in $G_{w'}$ of the chosen marking in $G_w$.  One can compute the finite
  orbit of $ (\calE_\calZ(w'))_\frakm, (\calE_\calP(w'))_\frakm$ by
  $\out{G_{w'}, \calE_{w'})}$.
	
\item 	For any marking  $ \calE_\calZ(w')_{\frakm'}, \calE_\calP(w')_{\frakm'}$ in this orbit, it is decidable whether there exists a  fibre-and-orientation preserving isomorphism  \[(G_w, (\calE_\calZ(w))_\frakm, (\calE_\calP(w))_\frakm) \to (G_{w'}, \calE_\calZ(w')_{\frakm'}, \calE_\calP(w')_{\frakm'}).\]     (We then say that  $ \calE_\calZ(w')_{\frakm'}, \calE_\calP(w')_{\frakm'}$ is an \define{admissible marking}).

\item  For all markings  $ \calE_\calZ(w')_{\frakm'}, \calE_\calP(w')_{\frakm'}$   an algorithm   computes a list  $\calL_{\frakm, \frakm'} $  of fibre-and-orientation preserving isomorphisms of the form \[(G_w, \calE_\calZ(w)_\frakm, \calE_\calP(w)_\frakm) \to (G'_{w'}, \calE_\calZ(w')_{\frakm'}, \calE_\calP(w')_{\frakm'}),\] such that  the following holds  :   

for    any  fibre-and-orientation preserving isomorphism \[\psi: (G_w,
  \calE_\calZ(w)_\frakm, \calE_\calP(w)_\frakm) \to (G'_{w'},
  \calE_\calZ(w')_{\frakm'}, \calE_\calP(w')_{\frakm'}),\] there
exists a resolving Dehn kernel $K'$ of $G_{w'}$, and $g\in G_{w'}$,
and an element $\phi\in \calL_{\frakm, \frakm'} $ such   that $\ad{g} \circ \psi  $ and $ \phi$ coincide in the quotient  $G_{w'}/K'$.\footnote{that is: for all
$h\in G_w$, there exists $z_h\in K'$ for which
$\psi(h)^g=\phi(h)z_h$.}
	
\end{itemize}
\end{prop}

Observe that in the third statement, the peripheral structures are marked. From the second point we already know that there exists an isomorphism from $G_w$ to $G_{w'}$ intertwining the markings, but we have no guarantee that this isomorphism can be taken fibre-and-orientation preserving.  Thanks to the third point we will know whether it can or not.

The first and last points of the statement of Proposition \ref{prop;double_DF} make a reformulation of  \cite[Proposition 2.4]{D_MSJ}, to which we refer, but that we don't syntactically reproduce.  This  \cite[Proposition 2.4]{D_MSJ} gives an algorithm that, given relatively hyperbolic  groups with residually finite peripheral subgroups (that fiber, endowed with a fiber and an orientation), without any non-trivial peripheral splitting (relative to their transversal peripheral structures),  terminates and provides a certificate of non-isomorphy, or a collection  of isomorphisms as in the last point, or a certain  cyclic splitting of one of the groups. We note that the argument of the proof of \cite[Proposition 2.4]{D_MSJ} does not require the fiber $F$ to be free, nor finitely generated. The crucial properties are the existence of the fiber, a well-defined orientation,  that the suspension is relatively hyperbolic relative to residually finite peripheral subgroups, and that there are no peripheral edge groups in its JSJ decomposition.
In our case, there are no non-trivial peripheral splittings  or cyclic splittings   (relative to their transversal peripheral structures), because we consider white vertex groups in a JSJ decomposition. 

\begin{proof}
The first assertion  is thus ensured by \cite[Proposition 2.4]{D_MSJ}.  
The second  is the previous Proposition \ref{prop;orbit_m}.
The third assertion is an application of  \cite[Proposition 5.5]{DG_Duke}
used with the constraint that the markings must be intertwined, as we
detail now. If there is no such isomorphism, \cite[Proproposition 5.5]{DG_Duke} ensures that,
in some characteristic 
  fibered Dehn fillings
$\overline{G_w}^{(m,f)}, \overline{G_{w'}}^{(m,f)}$   of $G_w$ and
$G_{w'}$ (with same $m$),  
 this will be apparent.  

However, in such a Dehn filling, all parabolic subgroups
have become virtually cyclic, and therefore, by the Dehn filling
theorem, the Dehn filling is word-hyperbolic. 

Therefore, for a fixed pair of corresponding fibered Dehn fillings  as
above,
 by the
solvability of the isomorphism problem for hyperbolic groups with marked peripheral structure
\cite[Thm 8.1]{DG_gafa_2010}, it is decidable whether or not there is  such isomorphism
between these Dehn fillings. 

 Enumerating the characteristic Dehn fillings in the fibre, and checking for each corresponding pairs of them this absence allows to detect if there exists indeed a pair for which there is no isomorphism. On the other hand, if there is an isomorphism of the correct form between $G_w$ and $G_{w'}$, it will be found by enumeration.
 
The fourth point is again treated in \cite[Proposition 2.4]{D_MSJ}.

\end{proof}

\begin{prop}\label{prop;the_white_album}
Consider $w$ a white vertex in the graph of groups $\bbJ_\alpha$, and
$w' = \phi_J(w)$, and $G_w$ and  $G_{w'}$ their groups with (unmarked ordered) peripheral
structures  $\calP(w), \calE_\calZ(w),
\calE_\calP(w)$, and $\calP(w'), \calE_\calZ(w'),
\calE_\calP(w')$.  One can compute a finite list $\calL_w$ that
 \begin{itemize} 
 \item is empty if and only if 
    \[(G_w, \calP(w), \calE_\calZ(w),
        \calE_\calP(w)), \hbox{ and } (G_{w'}, \calP(w'), \calE_\calZ(w'),
        \calE_\calP(w'))\] 
    are not isomorphic by a fibre and orientation
    preserving isomorphism,
  \item contains  fibre and orientation
preserving  isomorphisms 
    \[(G_w, \calP(w), \calE_\calZ(w),
        \calE_\calP(w)) \to (G_{w'}, \calP(w'), \calE_\calZ(w'),
        \calE_\calP(w'))\]
    \item is such that, for all isomorphisms \[\psi: (G_w, \calP(w), \calE_\calZ(w),
\calE_\calP(w)) \to (G_{w'}, \calP(w'), \calE_\calZ(w'),
\calE_\calP(w')),\] there is $\phi$ in $\calL_w$, $g\in G_{w'}$ and a
resolving Dehn kernel $K'$ of $G_{w'}$  such that for all $h\in G_w$,
$ \psi^g (h)\in \phi(h) K'$.
\end{itemize}
  \end{prop}

\begin{proof}
  By Proposition \ref{prop;double_DF} one can decide whether the two
  groups with (unmarked ordered) peripheral structures
  \[(G_w, \calP(w), \calE_\calZ(w), \calE_\calP(w))\;  \hbox{ and } \;
  (G_{w'}, \calP(w'), \calE_\calZ(w'), \calE_\calP(w'))\] are fibre
  and orientation preserving isomorphic or not, and if so, once a
  marking $\mathfrak m $ on $( \calE_\calZ(w), \calE_\calP(w))$ is
  chosen, one can compute all admissible markings (in the sense of
  Proposition \ref{prop;double_DF}) $\mathfrak m'$ on
  $(\calE_\calZ(w'), \calE_\calP(w'))$.  Now
  $\calL_w = \bigcup_{\mathfrak{m}'} \calL_{\mathfrak{m},
    \mathfrak{m}'}$, whose computability is given by Proposition
  \ref{prop;double_DF}. It clearly satisfies the two first points of
  the conclusion of Proposition \ref{prop;the_white_album}.  The last
  point of the Proposition \ref{prop;double_DF} ensures the third
  point, and thus proves Proposition \ref{prop;the_white_album}.

\end{proof}

\subsection{On black vertices: Navigating the Mixed Whitehead Problem}

Recall that we are considering a fixed isomorphism $\phi_J: J_\alpha \to J_\beta$ of the graphs
underlying the JSJ decompositions of $G$ and $G'$. We first mark the images of edge
groups in white vertices of $G$. We make a list exhausting all
possible matching markings of the corresponding white vertices in $G'$
(by isomorphisms preserving fibre and orientation). We pull back these
markings on the edge groups (through the attachment maps toward the white groups), and then we push the markings on the black groups, through the attachment maps of the reversed orientation edges,  
in $G'$ and in $G$.

We thus have, for each choice of white
matching marking, several markings of each black vertex, coming from
the adjacent white vertices. Choosing an order between neighbors, that
is matched by the graph isomorphism, we create the ordered tuple of
these markings.  Call these tuples compounded markings. We have them
for each black vertex of $G$, of $G'$, for each choice of white
matching marking.  We are moreover given  an unmarked isomorphism
between the matched black vertices, that preserves fibre and
orientation, by assumption (2) in the statement of Theorem
\ref{thm;IP_fop}.  We want to know whether one can post-compose these
isomorphisms with automorphisms of the black vertices that preserve
fibre and orientation, and that send the image of the compounded
marking of each $G_b$ to the compounded marking of $G_{b'}$. This is
exactly the fibre-and-orientation preserving mixed Whitehead problem,
as explained in the next proposition.

If $b$ is a black vertex, we denote its image
$\phi_J(b)$ by $ b'$. By assumption, we can decide whether there exists a fibre-and-orientation preserving isomorphism between $G_b$ and $G_{b'}$. If there is none, we cannot promote $\phi_J$ into an isomorphism of graphs of groups.

In the following, we will assume there is at least one such
isomorphism and we shall denote it $\phi_b^{(0)}$.

 \begin{prop}\label{prop;Black_album}
   Assume the fibre preserving mixed Whitehead problem has a solution
   for groups in $\calP_\beta$.

   Then there exists an algorithm such that, given $\phi_b^{(0)}$, and
   given, for each edge $e_j= (b,w_j)$ adjacent to $b$, a fibre
   preserving, and peripheral structures preserving isomorphism
   $\phi_{w_j}: G_{w_j} \to G_{w'_j}$, will indicate whether there is
   an isomorphism $G_b \to G_{b'}$ that preserves the fibre, the
   orientation, the unmarked ordered peripheral structures, and that
   peripherally coincides with
   $ t_{e_j'} \circ t_{\overline{e_j'}}^{-1} \circ \phi_{w_j} \circ
   t_{\overline{e_j}} \circ t_{e_j}^{-1}$ (i.e.  that, for all $j$,
   in restriction to $t_{e_j} (G_{e_j})$,   coincide to this map
   postcomposed by a conjugation by an element $p_{e_j}$ of $G_{b'}$).

   \end{prop}
\begin{proof}
For each edge group $G_e=G_{\bar e}$, take a generating set: it defines a marking of the attachement subgroups of adjacent vertex groups. Given $\phi_b^{(0)}$, one  looks for a fibre and orientation
preserving  automorphism of $G_{b'}$ that sends, for each adjacent
edge $e'_j$,   the conjugacy classes of the marking  of
$t_{e_j} (G_{e_j})$,  on the marking of   $ t_{e_j} (
i^{-1}_{\bar{e}'_j}  (  \phi_{w_j} (t_{\bar e_j}(G_{\bar{e_j}} )) ))
$.

Therefore, the problem to solve  is a reformulation of the fibre-and-orientation preserving mixed Whitehead problem in $G_{b'}$, which is given by assumption.
\end{proof}

   \subsection{Parts of the fibre in vertex groups}

Recall that we want to determine whether there exists an isomorphism $\torus \alpha \to \torus  \beta$ that is fibre and orientation preserving.

Recall that $\bbJ_\alpha$ and $\bbJ_\beta$ are the canonical JSJ
decompositions of $\torus \alpha$ and $\torus \beta$.  We say that an
isomorphism of graphs of groups $\Psi: \bbJ_\alpha \to \bbJ_\beta$,
noted $\Psi= (\psi_J, (\psi_v)_v, (\psi_e)_e, (\gamma_e)_e)$ (for
$\gamma_e \in G_{\tau(\psi_J(e))}$) is \define{fibre and orientation
  preserving for the vertices}, if  each $\psi_v$ is fibre
preserving, and at least one of $\psi_v$ or $\psi_e$ is orientation
preserving (it is easy to see that it forces all $\psi_v$ to be
orientation preserving).

In spite of these conditions, it is not automatic that the whole fibre
is preserved by $\Psi$.  This motivates the following definition.

Recall that we introduced, in Section \ref{sec;autom_gog}, the
definition of small modular group, as a subgroup of the automorphism
group of a graph of groups. We say that an isomorphism $\Psi$ as
described above is \define{fibre-controlled} 
if there exists an isomorphism
$\Psi^{(0)}= (\psi_J, (\psi^{(0)}_v), (\psi^{(0)}_e),
(\gamma^{(0)}_e))$ from $\bbJ_\alpha$ to $\bbJ_\beta$, that is the
composition of a fibre and orientation preserving isomorphism, with an
element of the small modular group of $\bbJ_\beta$, and such that,
there is a resolving Dehn filling of $\torus \beta$ in which the image
of $\psi_v$ and $\psi^{(0)}_v$ coincide,   
 and the images of $\psi_e$
and $\psi^{(0)}_e$ coincide too. Thus, in general the collections
$(\psi^{(0)}_v), (\psi^{(0)}_e)$ will be different from
$(\psi_v), (\psi_e) $, but will agree in a resolving Dehn filling.

\begin{prop}\label{prop;output_controlled_in_vertices}
  There exists an algorithm that, given $\torus \alpha, \torus \beta$,
  and $\bbJ_\alpha, \bbJ_\beta$ as above, terminates and with an
  output $\calO$ such that

\begin{itemize}
\item if there exists an isomorphism $\torus \alpha \to \torus \beta$
  that is fibre and orientation preserving, the algorithm outputs
  $\calO$, a non-empty finite collection of isomorphisms of graphs of
  groups $\bbJ_\alpha \to \bbJ_\beta$, that are fibre and orientation
  preserving for the vertices. Furthermore, at least one isomorphism
  in $\calO$ is fibre-controlled.

\item if there no such fibre and orientation preserving isomorphisms
  $\torus \alpha \to \torus \beta$, but if there exists
  $\Psi: \bbJ_\alpha \to \bbJ_\beta$ that is fibre and orientation
  preserving for the vertices, the algorithm outputs $\cal O$, a
  non-empty finite collection of such isomorphisms of graphs of
  groups.
  
\item If there is no $\Psi: \bbJ_\alpha \to \bbJ_\beta$ that is fibre
  and orientation preserving for the vertices, the algorithm outputs
  $\calO$,
  the empty set.
\end{itemize}

\end{prop}

\begin{proof}

Again we work with a fixed  graph isomorphism $\psi_J: J_\alpha \to J_\beta$ between the underlying graphs $J_\alpha, J_\beta$. 

Let $\calW$ be the set of white vertices of $\bbJ_\alpha$, and $\calB$
the set of black vertices of  $\bbJ_\alpha$.

For all white vertices $w\in \calW$, we use Proposition
\ref{prop;the_white_album} to get lists $\calL_w$ of isomorphisms from
$G_w$ to $G_{\psi_J(w)}$ that are fibre, orientation, and peripheral structure
preserving, and that satisfy the conditions of the third point of
Proposition \ref{prop;the_white_album}.   If one of them is empty,
 we may discard $\psi_J$, hence  we  assume that all are non-empty.

For every $b\in \calB$, one may decide whether there exists
a fibre-and-orientation preserving isomorphism $\phi_b^{(0)} : G_b \to
G_{\psi_J(b)}$. If for some $b$ there is none, we may discard
$\psi_J$, otherwise, we compute such an isomorphism for each $b\in \calB$.

For every tuple $(\phi_w)_{w\in \calW} \in \prod_{w\in \calW}
\calL_w$, we use Proposition \ref{prop;Black_album} for each black
vertex in $\calB$  to establish whether there exists, for each black
vertex $b$, a fibre and orientation preserving isomorphism $G_b\to
G_{b'}$ that agrees with the marking of the neighboring white vertices
(in the sense of Proposition \ref{prop;Black_album}).

If, given $(\phi_w)_{w\in \calW} \in \prod_{w\in \calW} \calL_w$, for
some $b\in \calB$, it is revealed that it is impossible to find such
an isomorphism, then again we may discard this tuple. Otherwise, we
compute such isomorphisms $\phi_b: G_b\to G_{\phi_J(b)}$ for each
$b\in \calB$.

Assume that one has found, for the given $\psi_J$, and a given $(\phi_w)_{w\in \calW} \in \prod_{w\in \calW}
\calL_w$,  isomorphisms
 $(\phi_b)_{b\in \calB}$ as above. Then one has a
complete collection, $(\phi_v)_v$ of isomorphisms of vertex
groups. For each edge $\{e, \bar e\}$, selecting an orientation
 unambiguously fixes $\tau(e)$, and by restriction of $\phi_{\tau(e)}$   one may define
$\phi_e$. One also  defines $\gamma_e$ for each $e$ such that $\tau(e) =
b\in \calB$
as the element $p_e$ given by Proposition \ref{prop;Black_album},
and for $\tau(e) =w \in \calW$, as an element that conjugates $\phi_w
(t_e(G_e))$ to $t_{\psi_J(e)} (G_{\psi_J(e)})$ (which exists since
$\phi_w$ preserves the ordered peripheral structure).  

One thus obtains that  $(\psi_J,
(\phi_v), (\phi_e), (\gamma_e)   )$ is a graph of groups isomorphism,
as it satisfies the commutativity of Bass diagram  (see Section \ref{sec;autom_gog}).

The algorithm then either has discarded $\psi_J$, or, for each tuple $(\phi_w)_{w\in \calW} \in \prod_{w\in \calW}
\calL_w$,  outputs  a graph
of groups isomorphism $(\psi_J,
(\phi_v), (\phi_e), (\gamma_e)   )$ for this 
$\psi_J$, if one exists.  
Thus the eventual output $\calO$ of the algorithm is a finite collection
of such isomorphisms, for $\psi_J$ and for $(\phi_w)_{w\in \calW} \in \prod_{w\in \calW}
\calL_w$ ranging over the set of graph
isomorphisms for which the process has been completed. We call this
collection the \define{selected isomorphisms}.

We need three lemmas. The first is
an observation.

\begin{lem}\label{lem;Ononempty} Assume that there is $\Psi^{(0)} = ( \psi_J,  (\psi_v), (\psi_e), (\gamma_e^{(0)}))$ an isomorphism of graph of groups that preserves the fibre and the orientation in the vertices.
	
	 Then there exists, for each white vertex $w$, an  isomorphism
         $\phi_w \in \calL_w$,    that satisfies the conclusion of
         Proposition \ref{prop;the_white_album}   for $\psi=\psi_w$,
         and for each black vertex $b$,  an isomorphism  $\phi_b$
         as above, agreeing on its neighboring edge groups with the
         isomorphisms of its adjacent white vertices.  In particular, the output $\calO$ is non-empty.
\end{lem}
\begin{proof}
	  The existence of $\phi_w$ is precisely  Proposition \ref{prop;the_white_album}, and determines the edge isomorphisms $\phi_e$.  The isomorphisms $\phi_b$ can be taken to be the $\psi_b$ since $\phi_w$ agrees with $\psi_w$ on its adjacent edge groups.  Finally, the elements  $(\gamma_e)_e $ can be taken to be  $(\gamma_e^{(0)})_e $ since the Bass diagram (of Section \ref{sec;autom_gog}) depends only on the restriction of the isomorphisms to edge groups.   
	\end{proof}

The second is the following. It is similar to  \cite[Lemma 2.7,  2.8 ]{D_MSJ}.

\begin{lem} \label{lem;216} If  there is $\Psi$, an isomorphism of graphs of groups, that preserves the fibre and the orientation in the whole groups $\torus  \alpha$, $\torus  \beta$, then there is a selected isomorphism of graph of groups, that, not only preserves the fibre and orientation in the vertices, but also is   fibre-controlled.  
\end{lem}

\begin{proof}	
	Let $\Psi= ({\psi_J}, (\psi_v)_v, (\psi_e)_e, \gamma_e)$ as in the statement.  By the composition rule of isomorphisms of graphs of groups, we have \[\Psi^{-1} = ( \psi_J^{-1}, ( \psi_{{\psi_J}^{-1}(v)}^{-1})_{v}, (\psi_{{\psi_J}^{-1}(e)}^{-1})_{e},  (  \psi^{-1}_{{\psi_J}^{-1}(\tau(e))}  ( (\gamma_{{\psi_J}^{-1}(e)})^{-1} )    )_e ).\] 
	
	By construction of $\calO$, there is a selected isomorphism
        $\Phi = (\psi_J, (\phi_v)_v, (\phi_e)_e, \gamma'_e)$, and a
        resolving Dehn filling of $\torus \beta$, such that
         in this Dehn filling, all $\psi_v$
        and $\phi_v$ coincide, up to conjugation in
        $\overline{G_{{\psi_J}(v)}}$, and moreover, $\Phi$ (as any element of
        $\calO$) preserves the fibre and the orientation in the
        vertices.
	
	Consider the composition $\Phi \circ \Psi^{-1}: \torus \beta
        \to \torus \beta$. It is an automorphism of $\bbJ_\beta$, and
        it can be written as  \[\Phi \circ \Psi^{-1} =(Id_{J_\beta},
        (\epsilon_v\circ \ad{g_v})_v ,   (\epsilon_e\circ \ad{g_e})_e
        ,    (\eta_e)_e   )\]
        in which $\epsilon_v$ is a fibre-and-orientation preserving automorphism of $G_v$, that
        induces the identity in $\overline{G_v}$, and $g_v$ is an
        element of $G_v$, and similarily for $\epsilon_e$ and $g_e$.
        For the record,  $\eta_e =  \phi_{{\psi_J}^{-1} (\tau(e))}  (  \psi_{{\psi_J}^{-1}(\tau(e))}^{-1}
        ( (\gamma_{{\psi_J}^{-1}(e)})^{-1} )  ) \gamma'_{{\psi_J}^{-1}(e)}$, but
        this plays no role in the argument.

	We need to show that   some post-composition of this
        automorphism with an element of the small modular group is
        fibre and orientation preserving.  In other words, we need to
        show that there is a fibre and orientation preserving
        automorphism $\Upsilon$ such that $\Upsilon\circ  \Phi \circ \Psi^{-1}
        $ is in the small modular group.

We will consider $\Upsilon$ of the form $\Upsilon = (Id, (\epsilon_v^{-1})_v,
 (\epsilon_e^{-1})_e, (y_e)_e )$, where the $y_e$ are yet unknown. It suffices
then to show that there is such an automorphism with the $y_e$ such
that $\Upsilon$ preserves the fibre and orientation.

We begin to prove that there is such an automorphism $\Upsilon$ with all
$y_e$ in the fibre (compare to \cite[Lemma 2.7]{D_MSJ}).

$\Upsilon$ being an automorphism of graphs of groups, it is obvious that
each $\epsilon_v^{-1}$  sends each adjacent edge group to a conjugate
of itself, by a conjugator in $G_v$.  Since $\epsilon_v^{-1}$ induces
the identity on a resolving Dehn filling,  such a  conjugator must be
in the pre-image of the centralizer of the image of the edge
group. However, because the Dehn filling is resolving, the image of
the centralizer of the edge group  in  $G_v$ is the centraliser of the
image of the edge group in $\overline{G_v}$. The centralizer of the
edge group  in  $G_v$ is not other than the edge group itself in
$G_v$, therefore, the conjugator that we considered is a product of an
element of the
edge group  in  $G_v$ and an element of the Dehn kernel. 
Therefore, it can be chosen in
the Dehn kernel, hence in the fibre.

Now, it suffices to choose the $y_e$ to
be these conjugating elements that send the edge groups to their images by
$\epsilon_v^{-1}$, and are in the fibre,   and then to choose $\epsilon_e$ the induced
automorphism of the edge groups,  in order to make
the Bass diagrams commute. Thus, the elements $y_e$ are in the fibre.

We  check that this choice is fibre and orientation
preserving.   (Compare to \cite[Lemma 2.8]{D_MSJ})

To see this, let the fibre  $F$ act on the tree $J_\beta$. This action
defines a free decomposition of $F$ as a graph of groups, hence $F$ is
generated by its intersection with the vertex groups, and the Bass
generators (corresponding to edges of the graph of groups.  The
automorphism $\Upsilon$ preserve the cosets of the fibre in the vertex
groups, and sends the Bass generator $e$ to $y_{\bar e}^{-1} e y_e$.

Since $y_e$ and $y_{\bar e}$ are in the fibre, the image of $e$ is in
the same  coset of the fibre as $e$. It follows that the generators of
$F$ given by the graph of groups decomposition are sent in $F$, thus
$\Upsilon$ is fibre preserving.

Finally, it is orientation preserving because it is so on any edge
groups, so it must be the case for the whole group. The Lemma is proved.

\end{proof}

  \begin{lem}  If  there is no isomorphism of graphs of groups from
    $\bbJ_\alpha $ to   $ \bbJ_\beta$  that is   fibre and orientation
    preserving for the vertices, then the output $\calO$ is empty. 
  \end{lem}
  \begin{proof} By construction of the output $\calO$ (see before Lemma \ref{lem;Ononempty}), its elements are isomorphisms of graphs of groups that when restricted to white vertices are in lists $\calL_w$, hence fiber and orientation preserving for the vertices (by Proposition
  \ref{prop;the_white_album}). The lemma follows.
  \end{proof}
  
  The three lemmas above together ensure the proposition holds.
\end{proof}

\subsection{The rest of the fibre and completing the proof of Theorem
  \ref{thm;IP_fop}.  }

 We now finish the proof of Theorem \ref{thm;IP_fop}. We assume that
 the output of the algorithm of Proposition
 \ref{prop;output_controlled_in_vertices} is a non-empty set $\calO$
 of isomorphisms of graphs of groups, that are fibre and orientation
 preserving for the vertices.  However, we do not yet know whether we
 are in the first or the second case of the conclusion of the
 Proposition \ref{prop;output_controlled_in_vertices}.

 Consider $\{h_1, \dots h_s\}$, a generating set  of the fibre $F$ of $\torus \alpha$.

 \begin{prop} \label{prop;only_the_h_i}

 	There exists a fibre and orientation preserving automorphism from $\torus \alpha$ to $\torus \beta$  if and only if, there is an element $\Phi$ of $\calO$, and an element $\eta$ of  the small modular group of $\bbJ_\beta$ that sends $\Phi(h_1), \dots, \Phi(h_s)$ inside the fibre of $\torus \beta$, and such that 
 	    $\eta \circ \Phi$ is fibre and orientation preserving on the vertices. 
  \end{prop}

	\begin{proof}
If there is such an element $\Phi$, and such an $\eta$, then  $\eta \circ \Phi$ is fibre and orientation preserving, because the fibre in $\torus \alpha$ is generated by the fibre in the vertices, and the elements $h_1, \dots, h_s$, and all this generating set is indeed sent to the fibre of $\torus \beta$. 

Assume now conversely that there exists a fibre and orientation preserving automorphism from $\torus \alpha$ to $\torus \beta$. Then we are in the first case of   Proposition \ref{prop;output_controlled_in_vertices}. There exists a   fibre-controlled  isomorphism \[\Phi = (\phi_J, (\phi_v)_v, (\phi_e)_e, (\gamma_e)_e) \; \in \; \calO. \]    
By definition, there is also an  
isomorphism \[\Phi^{(0)} = (\phi_J, (\phi^{(0)}_v)_v, (\phi^{(0)}_e)_e, (\gamma^{(0)}_e)_e),\] for which $\phi_v$ and $\phi^{(0)}_v$  coincide on a resolving Dehn filling, 
and such that some post-composition with an element of the small modular group is fibre and orientation preserving.

Therefore, there exists $\eta$ in the the small modular group of $\bbJ_\beta$ that sends \[\Phi(h_1), \dots, \Phi(h_s)\] inside the fibre of $\torus \beta$, and such that 
 $\eta \circ \Phi$ is fibre and orientation preserving on the vertices.

\end{proof}

It follows from Proposition \ref{prop;only_the_h_i} that we reduced
the problem to deciding, for each $\Phi$ in $\calO$, whether there
exists an element $\eta$ in the small modular group of $\bbJ_\beta$,
sending $\Phi(h_1), \dots, \Phi(h_s)$ into the fibre of
$\torus \beta$.

Given $\Phi$ in $\calO$, the problem of deciding whether such an
element $\eta$ of the small modular group as described above exists is
treated by interpreting it in the cohomology $\mathbb{Z}$-module
$H^1(\torus \beta, \bbZ) = \mathrm{Hom}(\torus\beta,\bbZ)$: each
Dehn-twist over an edge of $\bbJ_\beta$ (i.e. each generator of the
small modular group for the generating family proposed in Section
\ref{sec;autom_gog}) 
acts as a transvection on the module. Let $\bar F$ be the image of $F$
in $H^1(\torus\beta,\bbZ)$. Because the fibre $F$ is the kernel of a
homomorphism $\torus\beta \onto \bbZ$, for all $g \in \torus\beta$,
$g\in F$ if and only if, $\bar g \in \bar F$.

Each image $\overline{\Phi(h_i)}$ is the sum of an element of $\bar F$ and
of a multiple of $\bar t$, the image of $t$. The transvections apply
some translation in projection on the line generated by $t$. The
existence of some element in the  small modular group sending all
$\overline{\Phi(h_i)}$ in $\bar F$ is therefore encoded in an explicit
system of linear diophantine equations, whose existence of solution is
thus decidable.  We thus have proved Theorem \ref{thm;IP_fop}.
 
 \section{Trivial outer automorphisms of free groups}\label{sec:application}

We end this paper with the simplest possible type of maximal polynomial growth subgroup. Not only does this serve as an illustration of our algorithmic properties, but even this simplest case  yields  new results.

In Corollary \ref{cor:hyp-triv-ip} for relatively hyperbolic groups, the result follows relatively easily since edge groups of the mapping tori are virtually $\bbZ^2$. In the case of automorphisms of free groups, edge groups of the mapping tori can be more complicated. Consider the following example.
\begin{example}
Write a free group as $F=\langle x_1,y_1 \rangle*M*F\langle x_2,y_2\rangle$ and write $M = \langle H_1,H_2 \rangle$ where $H_1,H_2$ are subgroups. Let $\alpha$ be an automorphism of $F$ that is pointwise inner on $M$ and that maps the subgroups $K_i=\langle x_i,y_i, H_i \rangle$ to themselves for $i=1,2$. Suppose that $x_i,y_i$ are mapped to sufficiently complicated elements of $K_i$ so that the JSJ decompositions of the mapping tori $K_i \rtimes_{\alpha|_{K_I}}\bbZ$ relative to the maximal polynomial growth sub mapping tori are trivial. If we consider the full mapping torus $F \rtimes_\alpha \bbZ$, we have that the maximal polynomially growing sub mapping torus of $\alpha$ is of the form $M\times\bbZ$, and it will be a parabolic type vertex group of the 
JSJ decomposition of $F \rtimes_\alpha \bbZ$. There will be two non-parabolic vertex groups, isomorphic to the mapping tori $K_i \rtimes_{\alpha|_{K_i}}\bbZ$. The edge groups will be isomorphic to $H_i\times \bbZ$.
\end{example}

Thus, for automorphisms of free groups, or of many-ended hyperbolic groups, there is less control over the types of edge groups that can arise in the JSJ decomposition of mapping tori and a result like Corollary \ref{cor;ung} proven below is what is needed. 

 \subsection{A criterion for congruences to effectively separate torsion.}

\begin{prop}\label{prop;effective_minkowsky}
  Let $G$ be a finitely presented group such that the following
  hold:
  \begin{itemize}
  \item $G$ is conjugacy separable,
  \item for any $\alpha \in \aut G$ such that the image
    $[\alpha] \in \out G $ is non-trivial and of finite order, there exists 
    $g_\alpha \in G$ such that $g_\alpha$ and $\alpha(g_\alpha)$ are
    non-conjugate in $G$, (i.e. $G$ has no \define{pointwise inner} finite-order outer
    automorphisms), and
  \item we are given
  a finite
    list $\{\alpha_1,\ldots,\alpha_k\} \subset \aut G$ containing a
    representative of the conjugacy class of every finite order
    element of $\out G$.
  \end{itemize}
  Then $G$ has effective congruences separating torsion.
\end{prop}
\begin{proof}
  Let $\{\alpha_1,\ldots,\alpha_k\} \subset \aut G$ be a list of
  representatives of the conjugacy classes of finite order elements in
  $\out G$. Since $G$ is finitely presented, we can enumerate $G$ as
  well as all its finite quotients.  
  For each $\alpha_i$ in
  our list, by this enumeration  we can 
  find some $g_{\alpha_i}\in G$ and a finite
  quotient of $G$ in which the images $\alpha_i(g_{\alpha_i})$ and
  $g_{\alpha_i}$ are not conjugate. It follows that we can find finite
  index characteristic subgroup $K\leq G$ in which the image
  $\bar \alpha_i$ in $\aut{G/K}$ of each automorphism in our list is
  not inner, and the result follows.
\end{proof}

Finitely generated free groups are conjugacy separable, by a result of Baumslag \cite{Baumslag}, and their non-trivial outer-automorphisms are never pointwise inner, by a result of Grossman \cite{Grossman}. 
By Culler's Realization Theorem
\cite{Culler}, to obtain a list of all representatives for each finite
order outer automorphism, it is sufficient to enumerate the finite
number of homeomorphism types of graphs with vertices of degree at
least 3 whose fundamental group is $F_n$, and, for each such graph
graph, to enumerate the graph's symmetry group (which is finite). This
immediately gives.

\begin{cor}\label{cor;free_eff_minsk}
  $F_n$, the free group or rank $n$, has effective congruences separating torsion.
\end{cor}

\subsection{Algorithmic properties of mapping tori of unipotent non-growing automorphisms of free groups}\label{sec;non-growing}

 We finish by mentioning a direct application, which is formally stronger than Proposition \ref{prop;intro}.
 
 \begin{cor}\label{cor;ung}
 	Let $\calP$ be the class of direct products of finitely generated free groups with $\mathbb{Z}$, that we consider to be equipped with an explicit fiber and orientation. Then
 	\begin{enumerate}
 	    \item $\calP$ is hereditarily algorithmically tractable,
 	    \item the fiber an orientation preserving isomorphism problem is decidable with $\calP$,
 	    \item finitely generated subgroups of groups in $\calP$ have effective congruences separating torsion, and
 	    \item the fibre and orientation preserving mixed whitehead problem is solvable for arbitrary tuples of tuples for every group in $\calP$.
 	\end{enumerate} 
 	
\end{cor} 	

It is worth reminding that Corollary 3.3 is not true for torsion-free for hyperbolic groups by \cite{Rips}.

\begin{proof}
First, note that the uniform conjugacy problem, the uniform generation problem are classical in free groups, and in the class $\calP$. 

For groups in $\calP$, the fibre-and-orientation preserving isomorphism problem is solvable, since it is the question of the rank of the direct factor that is free. 

Take $G$ in the class $\calP$. Either $G \simeq \bbZ^2$ or there is a unique cyclic subgroup $C$ for which $G = H\times C$. Observe that $H$ is not unique, but all such factors are free subgroups of $G$ of same rank. Given $G$ in $\calP$ with a fibre and an orientation, the fibre-and-orientation preserving automorphism group is thus isomorphic to $\aut{H}$. The mixed Whitehead problem in $G$  by the fibre and orientation preserving automorphism group  then easily reduces to the mixed Whitehead problem in $H$, which is solvable by Whitehead's algorithm. 

It remains to see that in $\calP$ and its subgroups, congruences effectively separate the torsion. It is Minkowski's theorem if $G \simeq \bbZ^2$. Take $ G =F\times \bbZ$, where $F$ is a non-abelian free group.  

A non-trivial torsion element in $\out{G}$ either has even order, inducing  a flip of orientation of the center, or  maps to a non-trivial torsion element in $\out{F}$ by quotient by the characteristic subgroup $\bbZ$ (the center).  Indeed, assume the contrary. Write $G=F\times C$ with $C=\bk c$ the infinite cyclic center, and $\alpha \in \aut{G}$ such that $\bar \alpha : H\to H$ is inner: $\bar \alpha = \ad{h_0}$.  Since $c$ is central, for all $h\in H$, there is $n_h$ such that $\alpha(h)  = \ad{h_0} (h) c^{n_h}$,   and if $\alpha$ is not inner, there is $h$ such that $n_h\neq 0$, and assume it positive.   Also observe that $\alpha(c)= c^{\pm1}$.  If    $\alpha(c) = c^{-1}$, we are in the first case of the claim. Thus   assume that   $\alpha(c) = c$. It follows (still using that $c$ is central)  that $\alpha^k(h) = \ad{h_0^k} (h) c^{kn_h}$, and it is never inner, and thus $\alpha$ does not have finite order.  

Now, the free group $F$ itself has congruences effectively separating
the torsion by Corollary \ref{cor;free_eff_minsk}.  It follows that
one can effectively find congruence separating all non-trivial torsion
elements in $\out{G}$ except possibly those of even order inducing a
flip of orientation of the center. Those later ones are easily
separated, by choosing a finite quotient of $G$ on which the center
$C$ maps on a subgroup of order $\geq 3$.
\end{proof}

\noindent {\sc{Fran\c{c}ois Dahmani,   Institut Fourier,  Laboratoire de Mathématiques, Université Grenoble Alpes, CS 40700, 38058 Grenoble cedex 9, France. } }\\
e-mail. {\tt francois.dahmani@univ-grenoble-alpes.fr} \\  \url{https://www-fourier.univ-grenoble-alpes.fr/~dahmani}\\

\noindent {\sc Nicholas Touikan, Department of Mathematics and Statistics, University of New Brunswick, Fredericton, New Brunswick,  Canada.}\\
e-mail. {\tt ntouikan@unb.ca}\\
\url{https://ntouikan.ext.unb.ca}


{\small

\begin{thebibliography}{99}

\bibitem[Ba]{Bass} Hyman Bass, Covering theory for graphs of groups. J. Pure Appl. Algebra, 89(1- 2):3--47, 1993.

\bibitem[Bau]{Baumslag} Gilbert Baumslag, Residual nilpotence and relations in free groups. J. Algebra 2, 271--282, 1965.

\bibitem [Br]{Brinkmann} Peter Brinkmann,  Hyperbolic Automorphisms of Free Groups. Geom. Funct. Anal. 10 (2000), no. 5, 1071-1089.

\bibitem[BV]{BV2011} Oleg Bogopolski and Enric Ventura, On endomorphisms of torsion-free hyperbolic groups. Internat. J. Algebra Comput. 21 no. 8 (2011): 1415-1446.

\bibitem[C]{Culler} Marc Culler,  Finite groups of outer automorphisms of a free group. In {\it Contributions to group theory}. Contemp. Math. 33, 197-207 (1984). 
  
\bibitem[DFMT]{DFMT} François Dahmani, Stefano Francaviglia, Armando Martino, and Nicholas Touikan, The conjugacy problem in $\out{F_3}$, Forum  Math Sigma. 2025;13:e41   
  
  \bibitem[DG08]{DGr} Fran\c{c}ois Dahmani and Daniel Groves, The
    isomorphism problem for toral relatively hyperbolic groups,  Publ. Math. Inst. Hautes \'Etudes Sci . 107 no.1 (2008) 211-290.
\bibitem [DG10]{DG_gafa_2010} Fran\c{c}ois Dahmani and Vincent Guirardel, The isomorphism problem for all
  hyperbolic groups, Geom. Funct. Anal. 21 no.2 (2011) 223-300.
\bibitem [D16]{D_tams} Fran\c{c}ois Dahmani, On suspensions and conjugacy of hyperbolic
  automorphisms, Trans. Amer. Math. Soc. 368 (2016) 5565-5577.
\bibitem [D17]{D_MSJ} Fran\c{c}ois Dahmani, On suspensions and conjugacy of a few more
  automorphisms of free groups.  In {\it Hyperbolic Geometry and Geometric Group Theory,  Proceedings of the 7th Seasonal Institute of the Mathematical Society of Japan, K. Fujiwara et al. (ed.)}. Adv. Stud. Pure Math. 73 (2017)
 135-158.

\bibitem [DG18]{DG_Duke} Fran\c{c}ois Dahmani and Vincent Guirardel,    Recognizing a relatively hyperbolic group
  by its Dehn Fillings, Duke Math J. 167 no. 12. (2018) 2189-2241

\bibitem [DT19]{DT_invent_2019} Fran\c{c}ois Dahmani and  Nicholas Touikan, Deciding isomorphy using Dehn Fillings:
  the splitting case, Invent. Math. 215, no.1, (2019) 81-169.
  
\bibitem[DT23]{DT23}  Fran\c{c}ois Dahmani and  Nicholas Touikan, Unipotent linear suspensions of free groups, arXiv:2305.11274.

  
\bibitem [DL]{DLi}   Fran\c{c}ois Dahmani and  Ruoyu Li, Relative hyperbolicity for automorphisms of free
  products, J. Topol. Anal. 14, No. 1, 55-92 (2022).  
  
\bibitem [DK]{DK}  Fran\c{c}ois Dahmani and  Suraj Krishna, 
Relative hyperbolicity of hyperbolic-by-cyclic
  groups, Groups Geom. Dyn. 17, No. 2, 403-426 (2023).  
 
 \bibitem[FM]{FrancavigliaMartino} Stefano Francaviglia,  and Armando Martino,  Displacements of automorphisms of free groups. {II}: {Connectivity} of level sets and decision problems, Trans. Amer. Math. Soc. 375 (2022), no. 4, 2511--251.
 
  \bibitem[FH]{FH}  Mark Feighn and Michael Handel,  The Conjugacy Problem for UPG elements of $\out{F_n}$, Geom. Topol. 29, No. 4, 1693-1817 (2025).    
  
\bibitem[GL]{GL} Fran\c{c}ois Gautero, Martin Lustig, The mapping-torus of a free group automorphism is hyperbolic relative to the canonical subgroups of polynomial growth, preprint, arXiv:0707.0822.
  
\bibitem [Gh]{Ghosh} Pritam Ghosh, Relative hyperbolicity of free-by-cyclic extensions, Compos. Math. 159, No. 1, 153-183 (2023). 

\bibitem[Gr]{Grossman} Edna Grossman,  On the residual finiteness of certain mapping class groups. 
J. Lond. Math. Soc., II. Ser. 9, 160-164 (1974). 

\bibitem [GM]{GM_DFandES} Daniel Groves and Jason Manning, Dehn fillings and elementary splittings, Trans. Amer. Math. Soc. 370 (2018), 3017-3051.

\bibitem[GL]{GL_JSJ} Vincent Guirardel and Gilbert Levitt, JSJ  decompositions of groups. Ast\'erisque No. 395 (2017), vii+165 pp.
  
   
  \bibitem[Le05]{Lev05} Gilbert Levitt, Automorphisms of hyperbolic groups and graphs of groups, Geom. Dedicata 114 (2005), 49–70
   
  \bibitem[Le09]{Lev09} Gilbert Levitt, Counting growth types of automorphisms of free
    groups, Geom. Funct. Anal. 19, 1119 (2009).
    
    
    
    
\bibitem[O04]{Osin} Denis Osin, Relatively hyperbolic groups:
  intrinsic geometry, algebraic properties, and algorithmic
  problems. Mem. Amer. Math. Soc. 179, vi+100 (2006)    
\bibitem[O07]{Osin_filling} Denis Osin, Peripheral fillings of relatively hyperbolic groups, Invent. Math. 167 (2007), no. 2, 295–326. 

\bibitem[P]{Papasoglu} Panos Papasoglu, An algorithm detecting hyperbolicity. In {\it Geometric and computational perspectives on infinite groups (Minneapolis, MN and New Brunswick, NJ, 1994)}, v.25 DIMACS Ser. Discrete Math. Theoret. Comput. Sci. (1996) 193–200. 

\bibitem[R]{Rips} Eliyahu Rips, Subgroups of small cancellation groups, Bull. London Math. Soc. 14 (1982), no. 1, 45–47.



\bibitem[S]{Se} Zlil Sela, The isomorphism problem for hyperbolic
    groups I, Ann. of Math.  (2), 141 (1995), 217–283.
    

\bibitem[T]{Tou} Nicholas Touikan,  Detecting geometric splittings in finitely presented
  groups, Trans. Amer. Math. Soc. 370 (2018), 5635-5704 

\end{thebibliography}
  }
\end{document}